\documentclass[11pt, a4paper]{article}

\usepackage[english]{babel}
\usepackage{times}
\usepackage[letterpaper,top=2cm,bottom=2cm,left=3cm,right=3cm,marginparwidth=1.75cm]{geometry}
\usepackage{amsmath,amsfonts,amsthm,authblk,multirow}
\usepackage{empheq}
\usepackage[many]{tcolorbox}
\usepackage{dsfont}
\usepackage{graphicx}
\usepackage{float}
\usepackage{color,xcolor}
\usepackage{nicefrac}
\usepackage{todonotes}
\usepackage{tikz}
    \usetikzlibrary{arrows}
    \usetikzlibrary{decorations.pathreplacing}
\usepackage{stmaryrd}

\usepackage[colorlinks=true, allcolors=blue]{hyperref}
\usepackage[capitalise]{cleveref}

\newcommand{\p}{\partial}
\newcommand{\R}{\mathbb{R}}

\newtheorem{theorem}{Theorem}[section]

\newtheorem{remark}{Remark}[section]
\newcommand*\diff{\mathop{}\!\mathrm{d}}
\newtheorem*{remark*}{Remark}

\tikzset{every picture/.style={line width=0.75pt}}

\newtcolorbox{mybox}[1]{%
    tikznode boxed title,
    enhanced,
    arc=0mm,
    interior style={white},
    attach boxed title to top left= {yshift=-\tcboxedtitleheight/2-0.05cm, xshift=0.7cm},
    fonttitle=\small\bfseries,
    colbacktitle=white,coltitle=black,
    boxed title style={size=small,colframe=white,boxrule=0pt},
    title={#1}
    }

\title{ A Bound Preserving Energy Stable Scheme\\ for a Nonlocal Cahn--Hilliard Equation}
\author{Rainey Lyons\footnote{rainey.lyons@kau.se}, Grigor Nika, Adrian Muntean\\
Department of Mathematics and Computer Science,\\ Karlstad University, Sweden}

\date{\today}

\begin{document}

\maketitle

\begin{abstract}
\noindent
    We present a finite-volume based numerical scheme for a nonlocal Cahn--Hilliard equation which combines ideas from recent numerical schemes for gradient flow equations and nonlocal Cahn--Hilliard equations.
    The equation of interest is a special case of a previously derived and studied system of equations which describes phase separation in ternary mixtures.
    We prove the scheme is both energy stable and respects the analytical bounds of the solution.
    Furthermore, we present numerical demonstrations of the theoretical results using both the Flory--Huggins (FH) and Ginzburg--Landau (GL) free-energy potentials.
\end{abstract}



\section{Introduction}

Pattern and morphology formation are an important aspect of many areas of materials science especially, for example, in the construction of organic solar cells \cite{hoppe2004organic} and thin rubber bands \cite{creton2016rubber}. 
For these particular applications, the type of morphology produced greatly influences crucial properties of the studied materials, e.g., the efficiency of the solar cells or the mechanical behavior of the rubber bands.
These processes are often governed at the continuum level by a complex set of equations which usually involve an Allen--Cahn-type or Cahn--Hilliard-type equation tracking the time and space evolution of a phase indicator; see, for instance, \cite{miranville2019cahn} as well as  \cite{forest2023cahn,lyons2024phase,lyons2023continuum}. 
In this work we are interested in developing suitable numerical schemes approximating solutions to model equations that have the potential to govern morphology formation. 
In the applications above, the mixture considered is usually a blend of multiple solutes and some type (possibly multiple types) of solvent.
As a particular example, we consider the following system:

\begin{equation}\label{Eq:Marra}
    \p_t \begin{pmatrix}
        m\\
        \phi
    \end{pmatrix} = \nabla \cdot \left( \bar{M}(m,\phi)\nabla \frac{\delta \mathcal{F}}{\delta(m,\phi)}\right),
\end{equation}

where $\Bar{M}$ represents a degenerate mobility given by,

\begin{equation}\label{Eq:Marra_Mobility}
    \Bar{M}(m,\phi) := \beta (1-\phi) \begin{bmatrix}
    \phi + \frac{\phi^2 - m^2}{1-\phi} & m \\
    m & \phi
\end{bmatrix},
\end{equation}

and $\mathcal{F}$ is the free energy functional given by,

\[\mathcal{F}(m,\phi) = \int_\Omega g(m,\phi) \diff{x} + \frac{1}{2} \int_\Omega \int_\Omega J(x-x') [m(x) - m(x')]^2 \diff{x'} \diff{x}, \]

with 

\begin{align*}
g(m,\phi) 
&{:=} \phi - m^2 + \beta^{-1} \Big[ \frac{1}{2}(\phi+m)\log(\phi+m) + \frac{1}{2}(\phi-m)\log(\phi-m)\\
&+ (1-\phi) \log(1-\phi) - \phi \log(2)\Big].
\end{align*}

where the unknown $m: \Omega \longrightarrow [-1,1]$ represents the magnetization (or spin), while $\phi : \Omega \longrightarrow [0,1]$ represents the volume concentration density of solute.

Equation \eqref{Eq:Marra} was derived in \cite{Marra} as the rigorous hydrodynamic limit of the Kawasaki dynamics with inverse temperature $\beta >0$ for the Blume--Capel model with magnetic field $h_1$ and chemical potential $h_2$, whose Hamiltonian in a finite square $V\subset\mathbb{Z}^d$ is 

\begin{equation}\label{Eq_HamMarra}
H_\gamma (\sigma) 
= 
\frac{1}{2} \sum_{\substack{x \neq x' \in V}} 
J_\gamma (x-x')[\sigma(x)-\sigma(x')]^2 
-\sum_{x\in V}h_1 \sigma(x) 
-\sum_{x\in V}h_2 \sigma^2(x), 
\end{equation}
where $\sigma: \mathbb{Z}^d  \to \{-1,0,+1\}$ is the spin variable and $J_\gamma: \mathbb{R}^d\to \mathbb{R}$ is a Kac potential function with range of interaction $\gamma^{-1}$ such that 
\begin{equation}
\label{geigamma}
J_\gamma(r)=\gamma^d J(\gamma r)
\end{equation}
for all $r\in \mathbb{R}^d$ and a symmetric, compactly supported kernel $J\in C_+^2(\mathbb{R}^d)$ with the property $\int_{\mathbb{R}^d}J(r)dr=1$.

We refer the readers to \cite{Marra} as well as the recent works \cite{lyons2024phase,lyons2023continuum} for more information on the precise physical meaning of the variables at play and their relationship to those found in the Hamiltonian \eqref{Eq_HamMarra}.
This system has the potential to model morphology formation in the construction of organic solar cells and thin adhesive bands \cite{lyons2023continuum,lyons2024phase};
however, the numerical schemes used in previous works are finite-volume (FV) schemes based solely on the PDE formulation of the equation:
\begin{equation}\label{Eq:MarraPDE}
\left\lbrace
\begin{split}
    \p_t m &= \nabla \cdot \left[\nabla m - 2 \beta (\phi -m^2 ) (\nabla J * m) \right] \mbox{ in } (0,T)\times\Omega\\
    \p_t \phi &= \nabla \cdot \left[ \nabla \phi - 2 \beta m (1 - \phi) (\nabla J * m) \right]\mbox{ in } (0,T)\times\Omega
\end{split}\right.. 
\end{equation} 
While the use and study of FV schemes has many benefits such as conservation of mass, for this particular problem such schemes come with a few technical difficulties. For example, due to the sharp changes at the interface between phases, FV schemes require special techniques such as flux-limiter methods in order to preserve physically meaningful bounds (see, e.g., Figure \ref{fig:FV_Max} for an illustration on the importance of flux-limiter methods). 
Additionally, as these numerical methods were based on the strong formulation \eqref{Eq:MarraPDE}, the energy functional is only implicitly present in the scheme and it is unclear if the free energy is dissipated by the numerical scheme.
To this end, we aim to develop a numerical scheme which both preserves the analytical bounds of the solutions while also explicitly making use of the energy structure of the equation.

\begin{figure}[h]
    \centering
    \includegraphics[scale = 0.4]{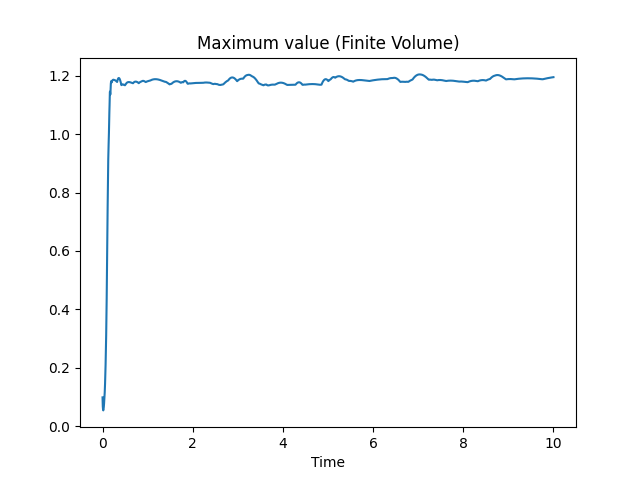}
    \caption{Maximum value of $|m|$ over time when solving \eqref{Eq:MarraPDE} with a finite-volume method without using flux limiters. Notice here that the analytical bound $|m| \leq 1$ is not maintained.}
    \label{fig:FV_Max}
\end{figure}

In this manuscript, we take a first step in this direction by considering a special case of model \eqref{Eq:Marra}, namely, the two phase scenario given when $\phi \equiv 1$. 
In this case, system \eqref{Eq:Marra} reduces down to a single equation of the form,

\begin{equation}\label{Eq:NLCH_GradFlow}
    \p_t \rho =  \nabla \cdot \Big[ M(\rho) \nabla \left( \frac{\delta \mathcal{E}(\rho)}{\delta \rho}\right)\Big].
\end{equation}
Here, to differentiate between settings, we use the variable $\rho$ instead of $m$ to represent the phase density. 
Heuristically, one could relate the mobilities of \eqref{Eq:NLCH_GradFlow} and \eqref{Eq:Marra} by $M(\rho) = \Bar{M}(\rho,1)$ and the free energies by $\mathcal{E}(\rho) = \mathcal{F}(\rho,1)$. 

While equation \eqref{Eq:Marra} is relatively novel, equation \eqref{Eq:NLCH_GradFlow} has been studied before in the literature. 
In \cite{Giacomin} this equation is rigorously derived using similar techniques and set up which lead to \eqref{Eq:Marra} and in \cite{giacomin1998phase} the sharp interface limit is explored.
There are also a plethora of numerical schemes for such equations as well, and the numerical scheme discussed in this paper, while novel to our knowledge, is a combination of the various aspects of these numerical methods.
We start first with the numerical methods studied in \cite{guan_diss,guan2014convergent,guan2014second}. 
Here, the authors approximate solutions to \eqref{Eq:NLCH_GradFlow} using a system of two discrete nonlinear equations, one describing the evolution of the phases and the other the evolution of the chemical potential, $\tfrac{\delta \mathcal{E}}{\delta \rho}$. 
By employing a convex splitting time discretization originally suggested in \cite{eyre1998unconditionally}, the authors of the aforementioned works are able to show strict dissipation of the discrete free energy or of a modified version of the free energy. 
While these schemes have seen great success, the addition of equations can be computationally taxing when considering systems of Cahn--Hilliard equations. 
Moreover, the performance of such schemes appear to be heavily dependent on the choice of nonlinear solver used.
On the other hand, it is possible to view equation \eqref{Eq:NLCH_GradFlow} from the perspective of one familiar with the field of conservation laws.
Such a philosophy is explored in many works including \cite{bailo2020fully,bailo2023,carrillo2015finite,huang2022sturcture}. 
These works are where we draw much of our inspiration and where we are closely related. 
In \cite{bailo2020fully,carrillo2015finite}, finite-volume based numerical schemes are developed for equations of similar structure to \eqref{Eq:NLCH_GradFlow} with nonlocal interactions in the free energy term, but with mobility $M(\rho) = \rho$ (or the constant mobility case) commonly seen in aggregation equations \cite{carrillo2014derivation}.
Meanwhile, the recent work \cite{bailo2023} employs similar ideas to the degenerate mobility case, but with only local terms in the free energy functional. 
To our knowledge, no one has yet applied these numerical techniques to equation \eqref{Eq:NLCH_GradFlow} with both nonlocal free energy and degenerate mobility. 

The paper is organized as follows: in Section \ref{Sec:Assumptions}, we discuss the assumptions used in throughout the manuscript; in Section \ref{Sec:Numerical_Scheme}, we introduce the numerical scheme and prove the main qualitative properties of the scheme; in Section \ref{Sec:Simulation}, we implement the numerical scheme and demonstrate the analytical results; finally in Section \ref{Sec:Conclusion}, we summarize the results and anticipate the difficulties in applying the numerical method to \eqref{Eq:Marra}.



\section{Model and assumptions}\label{Sec:Assumptions}

We are interested in developing suitable numerical schemes approximating solutions to model equations of the following form:

\begin{equation}
    \p_t \rho =  \nabla \cdot \Big[ M(\rho) \nabla \left( \frac{\delta \mathcal{E}(\rho)}{\delta \rho}\right)\Big],
\end{equation}

where $\rho: \Omega \longrightarrow [-1,1] $ represents the density of two distinct phases, $M(\rho) := \beta (1+\rho)(1-\rho)$ represents the degenerate mobility of the phases, and $\mathcal{E}$ is the free energy of the system with the form
\begin{equation}\label{Eq:Free_Energy}
    \mathcal{E}(\rho) := \int_\Omega  f(\rho(x)) \diff{x} + \frac{1}{2} \int_\Omega \int_\Omega J(x-x')[\rho(x) -\rho(x')]^2 \diff{x}\diff{x'}.  
\end{equation}

with  the Flory--Huggins free energy:

\begin{equation}\label{Eq:RefSystem_FH}
    f(\rho) = \beta^{-1} \left[(1-\rho)\log\left(\frac{1-\rho}{2}\right) + (1+\rho)\log\left(\frac{1+\rho}{2}\right) \right] + (1 - \rho^2).
\end{equation}

For generality, we assume that the energy landscape of the reference system $f$ is continuous and that there exists convex functions $f_c$ and $f_e$ such that $f = f_c - f_e$. 
This convex splitting will be utilized the time discretization in a similar way to \cite{bailo2023, eyre1998unconditionally, guan2014convergent}.  
This assumption fits to many well studied reference systems including the Flory--Huggins logarithmic potential \eqref{Eq:RefSystem_FH} and the Ginzburg--Landau double-well potential:

\begin{equation}\label{Eq:RefSystem_GL}
    f(\rho) = \frac{1}{4}(\rho^2 - 1)^2 = \frac{\rho^4 +1}{4} - \frac{\rho^2}{2}.
\end{equation} 

From the motivating set of equations \eqref{Eq:Marra}, we have that $J$ satisfies the following assumptions:
\begin{enumerate}
    \item[(J1)] $ J \in C^2(\R^d)$ is a nonnegative function compactly supported on the unit ball;
    \item[(J2)] $\int_{\R^d} J(x) \diff{x} = 1;$
    \item[(J3)] $J(x) = J(|x|)$, i.e., $J$ is radially symmetric.  
\end{enumerate}
We remark, however, that the analysis to follow does not require such smoothness on $J$. 
Indeed, the results below hold even in the case of $J \in W^{1,1}(\R^d)$.
Additionally, using the positive negative splitting studied in \cite{guan2014second}, one could relax the nonnegativity assumption as well.
Nevertheless, under these assumptions, one can show that 
\begin{equation}
    \mathcal{E}(\rho) = \int_\Omega  f(\rho(x)) \diff{x} + \int_\Omega \rho^2(x) \diff{x} - \int_\Omega \int_\Omega J(x-x') \rho(x') \diff{x'} \rho(x) \diff{x},
\end{equation}
and so, the chemical potential $w$ is given by

\begin{equation}
    w := \frac{\delta \mathcal{E}}{\delta \rho} = f'(\rho) + 2\rho - 2J*\rho.
\end{equation}

\section{Numerical Scheme}\label{Sec:Numerical_Scheme}
Throughout this section, we assume that equation \eqref{Eq:NLCH_GradFlow} is posed on a 2 dimensional square, $\Omega \subset \R^2$, equipped with periodic boundary conditions.
This is done purely for notational convenience and the extension to higher dimensions and no flux boundary conditions is straight forward (see, e.g., \cite{bailo2023} for a similar scenario). 

\subsection{Discretization}
 For lucidity, we assume that the square $\Omega$ is discretized by a uniform spacial mesh size, $h>0$. 
 We remark that a uniform spacial mesh size is not necessary, and the adjustments to the scheme can be easily implemented.
 We denote the uniformly spaced nodes of the mesh by the pair $(x_i,y_j)$ and we denote the mesh cells by $\Lambda_{i,j} := [x_i -\frac12 h, x_i +\frac12 h) \times [y_j -\frac12 h, y_j +\frac12 h) $ for $i,j = 0,1,2,\dots, N-1$. 
 To account for the periodic boundary conditions, we periodically extend functions defined on the nodes, i.e., $f(x_i,y_j) = f(x_{i\pm N},y_j) = f(x_i, y_{j\pm N})$. 
 We approximate given initial data, $\rho_0$, in the standard way 
\[ \rho_{i,j}^0 := \frac{1}{|\Lambda_{i,j}|} \int_{\Lambda_{i,j}} \rho_0(x,y) \diff{x}\diff{y}.\]
We denote the discrete inner product by $\langle \cdot, \cdot \rangle_h$ and the discrete $\ell^2$ norm by $\| \cdot \|_{\ell^2}$. For example, 
\[\| \rho^0 \|_{\ell^2}^2 := \langle \rho^0,\rho^0 \rangle_h = h^2 \sum_{i,j = 0}^{N-1} (\rho_{i,j}^0)^2.\]

\subsection{Circular Convolution}
For simplicity, we approximate the circular convolution in both schemes with the operation $\star$ defined as
\begin{equation}\label{Eq:CircConv}
    [J\star \rho]^k_{i,j} := h^2 \sum_{n,m = 0}^{N-1} J_{n,m} \,  \rho^k_{i-n,j-m},
\end{equation}
where $\rho$ has been periodically extended to make sense of unusual indices. 
We point out that other approximations are valid such as the edge-valued version of \eqref{Eq:CircConv} used in \cite{guan2014second,guan2014convergent} as well as the fast Fourier transform methods similar to \cite{TiwariKumaretal_2021_FastAccurateApproximation}. 
The main properties of the approximation needed in our analysis are 
\begin{equation}\label{Eq:CircConv_Commute}
    \langle [J \star \phi] , \psi \rangle_h = \langle [J \star \psi] , \phi \rangle_h  
\end{equation}
and 
\begin{equation}\label{Eq:CircConvBound}
    \left| \langle [J \star \phi] , \psi \rangle_h \right| \leq [J\star 1] \left( \frac{\varepsilon}{2} \| \phi \|_{\ell^2}^{2} + \frac{1}{2 \varepsilon} \| \psi \|_{\ell^2}^2 \right),
\end{equation}
whenever $\phi$ and $ \psi$ are periodic and $J$ is nonnegative and radially symmetric.
Therefore, any approximation of the convolution with these properties can be used. A proof of these inequalities for \eqref{Eq:CircConv} is similar to the proof found in \cite[Section 3.3]{guan_diss} and is therefore omitted.

\subsection{Bound Preserving Scheme}
Inspired by the works \cite{bailo2023,huang2022sturcture}, we consider a conservation law type scheme on equation \eqref{Eq:NLCH_GradFlow}. 
We adapt similar notation to that of \cite{bailo2023} for convenience to the reader. 
Keeping the same space and time discretization as before, we prescribe the scheme
\begin{equation}\label{Eq:CL_Scheme}
    \frac{\rho^{k+1}_{i,j} - \rho^{k}_{i,j}}{\Delta t} + \frac{1}{h}\left[F_{i+1/2,j}^{k+1} - F_{i-1/2,j}^{k+1} + F_{i,j+1/2}^{k+1} - F_{i,j-1/2}^{k+1}\right] =0,
\end{equation}
where the numerical flux $F$ is given by
\[ F^{k+1}_{i+1/2,j} = M(\rho^{k+1}_{i,j} , \rho^{k+1}_{i+1,j}) [u_{i+1/2,j}^{k+1}]^+ + M(\rho^{k+1}_{i+1,j} , \rho^{k+1}_{i,j} ) [u_{i+1/2,j}^{k+1}]^-\]
and 
\[F^{k+1}_{i,j+1/2} = M(\rho^{k+1}_{i,j} , \rho^{k+1}_{i,j+1}) [u_{i,j+1/2}^{k+1}]^+ + M(\rho^{k+1}_{i,j+1} , \rho^{k+1}_{i,j} ) [u_{i,j+1/2}^{k+1}]^-,\]
where the discrete mobility is calculated with the function 
\begin{equation}\label{Eq:Mobility_Split}
    M(x,y) := \beta [1+x]^+ [1-y]^+,
\end{equation}
and the velocity $u^{k+1}_{l+1/2}:= \tfrac{-1}{h}(w^{k+1}_{l+1} - w_l^{k+1})$ with
\[w^{k+1}_{i,j} := f_c'(\rho^{k+1}_{i,j}) - f_e'(\rho^{k}_{i,j}) + 2\rho^{k+1}_{i,j} - 2[J\star \rho]^k_{i,j}.\]  

One of the novelties of the work \cite{bailo2023} is the handling of the degenerate mobility \eqref{Eq:Mobility_Split}. 
As we will show later, this splitting is the main reason the analytical bound $|\rho| \leq 1$ is preserved by the numerical scheme. 

\subsection{Main results}
In this section, we provide proofs of the conservation of mass, boundedness, and energy stability properties of the scheme \eqref{Eq:CL_Scheme}. The proof of the following theorem is essentially the same as that found in Section 2.1 of \cite{bailo2023} with slight adjustments to account for periodic boundary conditions. 

\begin{theorem}
    For any $\Delta t, h > 0$, scheme \eqref{Eq:CL_Scheme} has the following properties:
    \begin{enumerate}
        \item \textbf{Conservation of mass:} $\sum_{i,j} \rho_{i,j}^{k+1} = \sum_{i,j} \rho_{i,j}^{k}$ ;
        \item \textbf{Boundedness of the phase-field:} $|\rho_{i,j}^k| \leq 1$ for all $i,j$ $\Rightarrow$ $|\rho_{i,j}^{k+1}| \leq 1$ for all $i,j$.        
    \end{enumerate}
\end{theorem}
\begin{proof}
    \textbf{Conservation of mass:} Summing the scheme \eqref{Eq:CL_Scheme} over $i,j = 0,\dots, N-1$ and abusing the telescopic sum, we have from periodicity
    \begin{align*}
       \sum_{i,j = 0}^{N-1}\frac{\rho^{k+1}_{i,j} - \rho^{k}_{i,j}}{\Delta t} &= - \sum_{i,j = 0}^{N-1}\frac{1}{h}\left[F_{i+1/2,j}^{k+1} - F_{i-1/2,j}^{k+1} + F_{i,j+1/2}^{k+1} - F_{i,j-1/2}^{k+1}\right]  \\
       &= -\frac{1}{h}\sum_{j=0}^{N-1}\left[F_{N-1/2,j}^{k+1} - F_{-1/2,j}^{k+1} \right] - \frac{1}{h}\sum_{i=0}^{N-1}\left[ F_{i,N-1/2}^{k+1} - F_{i,-1/2}^{k+1}\right] \\
       &= 0.
    \end{align*}

    \textbf{Boundedness:} Say for some $n,m \in \{0,1,\dots,N-1\}$ there is a point $\rho_{n,m}^{k+1} > 1$. Then
    \begin{align*}
        0 < h \frac{\rho_{n,m}^{k+1}-\rho_{n,m}^{k}}{\Delta t} &= -F_{n+1/2,m}^{k+1} + F_{n-1/2,m}^{k+1} - F_{n,m+1/2}^{k+1} + F_{n,m-1/2}^{k+1}  \\
         &= -M(\rho^{k+1}_{n,m},\rho^{k+1}_{n+1,m})[u^{k+1}_{n+1/2,m}]^+ - M(\rho^{k+1}_{n+1,m} , \rho^{k+1}_{n,m} ) [u_{n+1/2,m}^{k+1}]^- \\
         &\quad + M(\rho^{k+1}_{n-1,m},\rho^{k+1}_{n,m})[u^{k+1}_{n-1/2,m}]^+ + M(\rho^{k+1}_{n,m} , \rho^{k+1}_{n-1,m} ) [u_{n-1/2,m}^{k+1}]^- \\
         &\quad -M(\rho^{k+1}_{n,m} , \rho^{k+1}_{n,m+1}) [u_{n,m+1/2}^{k+1}]^+ - M(\rho^{k+1}_{n,m+1} , \rho^{k+1}_{n,m} ) [u_{n,m+1/2}^{k+1}]^- \\
         & \quad +M(\rho^{k+1}_{n,m-1} , \rho^{k+1}_{n,m}) [u_{n,m-1/2}^{k+1}]^+ + M(\rho^{k+1}_{n,m} , \rho^{k+1}_{n,m-1} ) [u_{n,m-1/2}^{k+1}]^- .
    \end{align*}
    Recalling the mobility $M(x,y) = \beta [1+x]^+ [1-y]^+$, we see that $M$ returns only nonnegative values. Furthermore, since $M(x,y) = 0$ when $y > 1$, the above inequality reduces to 
    \begin{align*}
        0 &< -M(\rho^{k+1}_{n,m},\rho^{k+1}_{n+1,m})[u^{k+1}_{n+1/2,m}]^+  -M(\rho^{k+1}_{n,m} , \rho^{k+1}_{n,m+1}) [u_{n,m+1/2}^{k+1}]^+ \\
        &\quad+ M(\rho^{k+1}_{n,m} , \rho^{k+1}_{n-1,m} ) [u_{n-1/2,m}^{k+1}]^-+ M(\rho^{k+1}_{n,m} , \rho^{k+1}_{n,m-1} ) [u_{n,m-1/2}^{k+1}]^- .
    \end{align*}
    However, we see that this is a contradiction. Indeed, each term in the above inequality is either zero or negative depending on the neighboring values $\rho_{n-1,m}^{k+1}$, $\rho_{n+1,m}^{k+1}$, $\rho_{n,m-1}^{k+1}$, and $\rho_{n,m+1}^{k+1}$. The proof for $\rho_{i,j}^{k+1} \geq -1$ follows a similar argument.

\end{proof}

Generally in the simulation of gradient flow equations, it is desirable for the numerical scheme to dissipate a discrete version of \eqref{Eq:Free_Energy}, namely, 
\begin{equation}\label{Eq:Discrete_FreeEnergy}
    \mathcal{E}_h (\rho^{k+1}) := h^2 \sum_{i,j = 0}^{N-1} [f_c(\rho^{k+1}) - f_e(\rho^{k+1}) + (\rho^{k+1})^2 ] - \langle J\star[\rho^{k+1}] , \rho^{k+1} \rangle_h.
\end{equation}
However, due to the complications introduced by the interaction between the choice of time stepping and the convolution, strict free energy dissipation of scheme \eqref{Eq:CL_Scheme} may be out of reach. 
However, by introducing the following pseudo energy:
\begin{equation}\label{Eq:PseudoEnergy}
    \hat{\mathcal{E}}_h (\rho^{k+1},\rho^{k}) := \mathcal{E}_h(\rho^{k+1}) + \|\rho^{k+1}-\rho^{k}\|_{\ell^2}^2 +  \langle J\star[\rho^{k+1}-\rho^{k}] ,\rho^{k+1}-\rho^{k}\rangle_h,
\end{equation}
we can arrive at a stability result.

\begin{theorem}
    For any $\Delta t, h > 0$, 
    \[ \hat{\mathcal{E}}_h (\rho^{k+1},\rho^{k}) - \hat{\mathcal{E}}_h (\rho^{k},\rho^{k-1}) \leq \|\rho^{k+1} - \rho^k \|_{\ell^2}^2.  \]
\end{theorem}
\begin{proof}
    We begin first by showing 
    
    \begin{equation}\label{Eq:NegativeInnerPotential}
        \langle \rho^{k+1}- \rho^k , w^{k+1} \rangle_h \leq 0.
    \end{equation}     
    Indeed, by using the structure of scheme \eqref{Eq:CL_Scheme} and the periodic boundary conditions, we have the following computation:
    
    \begin{align*}
    \langle \rho^{k+1}- \rho^k , w^{k+1} \rangle_h = 
        &- h \Delta t \sum_{i,j = 0}^{N-1}  (F_{i+ 1/2 , j}^{k+1} -F_{i- 1/2 , j}^{k+1} +F_{i,j+ 1/2 }^{k+1} -F_{i,j- 1/2 }^{k+1} ) w_{i,j}^{k+1} \\ 
        &\quad = - h\Delta t \sum_{i,j = 0}^{N-1} (F_{i+ 1/2 , j}^{k+1} -F_{i- 1/2 , j}^{k+1}) w_{i,j}^{k+1} 
        - h\Delta t \sum_{i,j = 0}^{N-1} (F_{i,j+ 1/2 }^{k+1} -F_{i,j- 1/2 }^{k+1}) w_{i,j}^{k+1} \\
        &\quad = - h\Delta t \sum_{i,j = 0}^{N-1} (  w_{i,j}^{k+1} - w_{i+1,j}^{k+1}) F_{i+ 1/2 , j}^{k+1} - h\Delta t \sum_{i,j = 0}^{N-1} (  w_{i,j}^{k+1} - w_{i,j+1}^{k+1}) F_{i , j+ 1/2}^{k+1} \\
        &\quad = -h^2 \Delta t \sum_{i,j = 0}^{N-1} u_{i+1/2,j}^{k+1} F_{i+ 1/2 , j}^{k+1} -h^2 \Delta t \sum_{i,j = 0}^{N-1} u_{i,j+ 1/2}^{k+1} F_{i , j+ 1/2}^{k+1}\\
        &\quad = -h^2 \Delta t \sum_{i,j = 0}^{N-1} u_{i+1/2,j}^{k+1} \left[M(\rho^{k+1}_{i,j} , \rho^{k+1}_{i+1,j}) [u_{i+1/2,j}^{k+1}]^+ + M(\rho^{k+1}_{i+1,j} , \rho^{k+1}_{i,j} ) [u_{i+1/2,j}^{k+1}]^- \right] \\
        &\qquad - h^2 \Delta t \sum_{i,j = 0}^{N-1} u_{i,j+ 1/2}^{k+1} \left[ M(\rho^{k+1}_{i,j} , \rho^{k+1}_{i,j+1}) [u_{i,j+1/2}^{k+1}]^+ + M(\rho^{k+1}_{i,j+1} , \rho^{k+1}_{i,j} ) [u_{i,j+1/2}^{k+1}]^- \right] \\
        &\quad \leq -h^2 \Delta t \sum_{i,j = 0}^{N-1}  \min\{M(\rho^{k+1}_{i,j} , \rho^{k+1}_{i+1,j}), M(\rho^{k+1}_{i+1,j} , \rho^{k+1}_{i,j} )\} |u_{i+1/2,j}^{k+1}|^2\\
        &\qquad -h^2 \Delta t \sum_{i,j = 0}^{N-1}  \min\{M(\rho^{k+1}_{i,j} , \rho^{k+1}_{i,j+1}), M(\rho^{k+1}_{i,j+1} , \rho^{k+1}_{i,j} )\} |u_{i,j+ 1/2}^{k+1}|^2 \leq 0 .
    \end{align*}

    Now, we have by unpacking definitions,
    
    \begin{align}
        \hat{\mathcal{E}}_h (\rho^{k+1},\rho^{k}) - \hat{\mathcal{E}}_h (\rho^{k},\rho^{k-1}) &= \underbrace{h^2 \sum_{i,j = 0}^{N-1} [f_c(\rho_{i,j}^{k+1})-f_c(\rho_{i,j}^k) - f_e(\rho_{i,j}^{k+1}) + f_e(\rho_{i,j}^{k}) + (\rho^{k+1}_{i,j})^2 - (\rho^k_{i,j})^2 ]}_{A} \notag\\
        &\quad \underbrace{- \langle J\star[\rho^{k+1}] ,\rho^{k+1}\rangle_h + \langle J\star[\rho^{k}] ,\rho^{k}\rangle_h}_{B}\notag\\
        &\quad+ \|\rho^{k+1}-\rho^{k} \|_{\ell^2}^2 -\|\rho^{k}-\rho^{k-1} \|_{\ell^2}^2\notag\\
        &\quad +\langle J\star[\rho^{k+1}-\rho^{k}] ,\rho^{k+1}-\rho^{k}\rangle_h - \langle J\star[\rho^{k}-\rho^{k-1}] ,\rho^{k}-\rho^{k-1}\rangle_h. \label{Eq:PsuEnergyDiff}
    \end{align}
    Looking first at the terms making up $B$, we have from algebraic manipulation
    
    \begin{align*}
        - \langle J\star[\rho^{k+1}] ,\rho^{k+1}\rangle_h + \langle J\star[\rho^{k}] ,\rho^{k}\rangle_h = - \langle J\star[\rho^{k+1}-\rho^{k}] ,\rho^{k+1}-\rho^{k}\rangle_h - 2\langle J\star[\rho^k], \rho^{k+1}-\rho^{k}\rangle_h.
    \end{align*}
    Turning now to the terms that make up $A$, we have via convexity of $f_c$, $-f_e$, and $z^2$,
    
    \begin{align*}
        A \leq \langle \rho^{k+1} - \rho^k , f'_c(\rho^{k+1}) - f'_e(\rho^k) + 2\rho^{k+1} \rangle_h.
    \end{align*}
    Applying the above results to \eqref{Eq:PsuEnergyDiff}, we have
    
    \begin{align*}
        \hat{\mathcal{E}}_h (\rho^{k+1},\rho^{k}) - \hat{\mathcal{E}}_h (\rho^{k},\rho^{k-1})  &\leq \langle \rho^{k+1}- \rho^k , w^{k+1} \rangle_h + \|\rho^{k+1}-\rho^{k} \|_{\ell^2}^2\\&
        \quad -\|\rho^{k}-\rho^{k-1} \|_{\ell^2}^2 - \langle J\star[\rho^{k}-\rho^{k-1}] ,\rho^{k}-\rho^{k-1}\rangle_h. 
    \end{align*}    
    Finally, owing to \eqref{Eq:NegativeInnerPotential} and \eqref{Eq:CircConvBound} we obtain the desired result.
\end{proof}

\begin{remark}
    Using similar arguments to the proof above, we can prove strict free energy dissipation for a semi-discrete (continuous in time) version of \eqref{Eq:CL_Scheme}, namely, we can prove:
\begin{theorem}
    Let the discrete free energy be given by 
    \[ \mathcal{E}_h (t) :=  h^2 \sum_{i,j = 0}^{N-1} [f(\rho(t)_{i,j})  + (\rho(t)_{i,j})^2 ] - \langle J\star[\rho(t)] , \rho(t)\rangle_h.\]
    Then, for any $h>0$,
    \[\frac{\diff}{\diff t}\mathcal{E}_h(t) \leq 0.\]
\end{theorem}
\begin{proof}
    Differentiating $\mathcal{E}_h$ with respect to time and owing to \eqref{Eq:CircConv_Commute}, we have
    \begin{align*}
        \frac{\diff}{\diff t}\mathcal{E}_h(t) &= h^2\sum_{i,j = 0}^{N-1}  f'(\rho(t)_{i,j}) \rho'(t)_{i,j} + 2\rho(t)_{i,j} \rho'(t)_{i,j} - 2 \langle J\star[\rho(t)] , \rho'(t)\rangle_h \\
        &= \langle \rho'(t), f'(\rho(t)) + 2\rho(t) - 2  J\star[\rho(t)] \rangle_h
        = \langle \rho'(t) , w(t) \rangle_h \leq 0.
    \end{align*}
    Where the last inequality follows similar arguments to those which lead to \eqref{Eq:NegativeInnerPotential}.
\end{proof}
    Therefore, it may be possible to examine other choice of time discretizations in a similar manner as \cite{bailo2020fully} and arrive at strict energy dissipation. 
    However, as the authors point out in the aforementioned work, these choices of discretization are computationally less efficient as they treat the convolution term implicitly.  
\end{remark}

\section{Simulation}\label{Sec:Simulation}

\begin{figure}[h]
    \includegraphics[width = 6in]{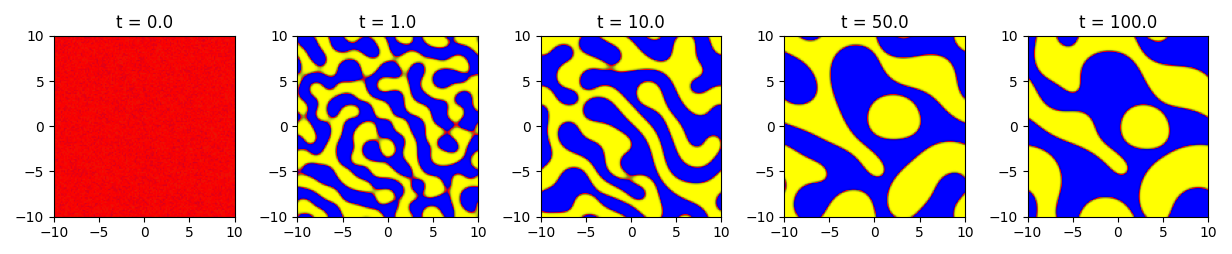}
    \includegraphics[width = 6in]{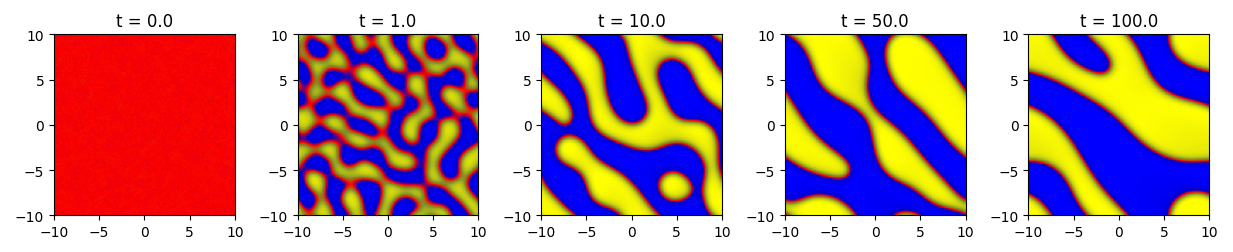}
    \caption{Simulation of phase separation with the Flory--Huggins potential (top) and Ginzburg--Landau potential (bottom) using scheme \eqref{Eq:CL_Scheme} for $N = 2^7$, $\beta = 5$, and with a time step $\Delta t = 10^{-2}$. Here, phases near 1 are colored blue, near -1 yellow, and near 0 red. }
    \label{fig:PhaseSeparation}
\end{figure}

In this section, we numerically demonstrate the properties proven in the previous section for both the Flory--Huggins (FH) potential \eqref{Eq:RefSystem_FH} and the Ginzburg--Landau (GL) potential \eqref{Eq:RefSystem_GL}. 
For these simulations, we take the initial condition to be a small random perturbation of the zero solution. 
This type of initial condition represents a `well-mixed' initial condition. 
In Figure \ref{fig:PhaseSeparation}, we plot the approximated solution to equation \eqref{Eq:NLCH_GradFlow} with periodic boundary conditions over time. 
In Figure \ref{fig:Max_Val+Energy}, we plot the evolution of the discrete free energy \eqref{Eq:Discrete_FreeEnergy} and the maximum value of $|\rho|$ over time.
Comparing these results to Figure \ref{fig:FV_Max}, we see that the simulation strictly adheres to the analytical bounds of the equation. 
Finally in Figure \ref{fig:Energy_Dissipation}, we display in a log-log plot the evolution of the discrete free energy over time compared to functions of the form $y=Ct^{p}$. 
We point out that the dissipation rate appears to be bounded between the powers $p = \frac{-1}{3}$ and $p = \frac{-1}{4}$ which is inline with previous results \cite{du2018stabilized, kohn2002upper}.

\begin{figure}[!htb]
    \centering
    \includegraphics[scale = 0.45]{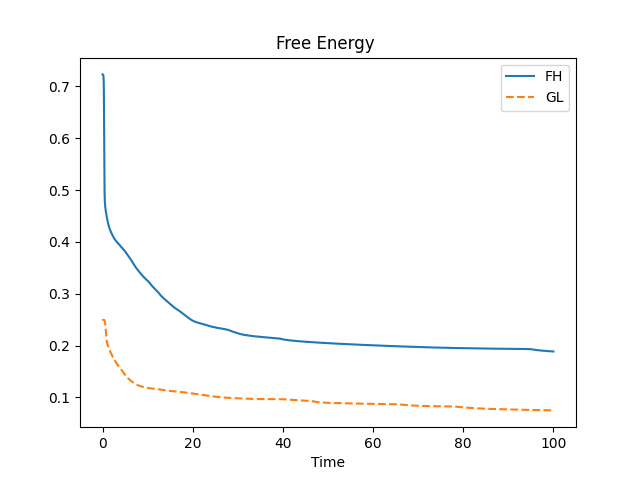}   
    \includegraphics[scale = 0.45]{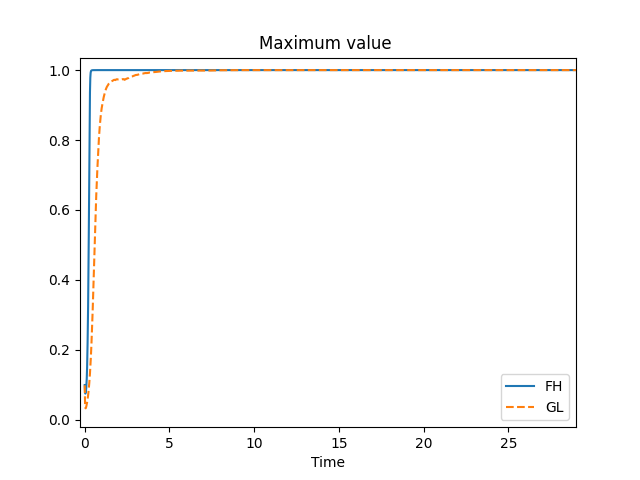}
    \caption{\textbf{Left:} Evolution of the free energy \eqref{Eq:Free_Energy} over time for both the Flory--Huggins (FH) and Ginzburg--Landau (GL) potentials.     
    \textbf{Right:} Maximum value of $|\rho|$ over time.}
    \label{fig:Max_Val+Energy}
\end{figure}

\begin{figure}[!htb]
    \centering
    \includegraphics[scale = 0.75]{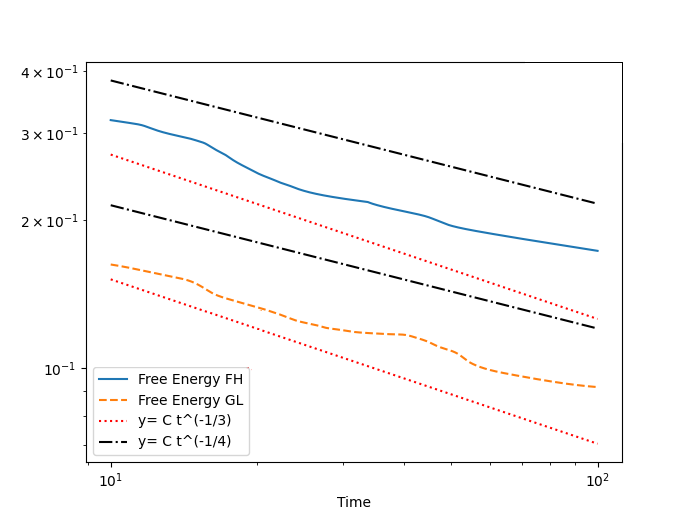}   
    \caption{Energy dissipation rate for the Flory--Huggins (FH) and Ginzburg--Landau (GL) potentials.}
    \label{fig:Energy_Dissipation}
\end{figure}

\section{Conclusion}\label{Sec:Conclusion}
In this work, we have presented a finite volumes based numerical method heavily inspired by the schemes used for gradient flow equations \cite{bailo2020fully,bailo2023,huang2022sturcture} with adjustments to handle the nonlocal terms inspired by \cite{guan_diss,guan2014convergent,guan2014second}.
We have shown this numerical scheme unconditionally conserves mass and preserves the analytical bounds of the solution.
Additionally, since this scheme is based on the gradient flow formulation of the equation, we can show the scheme is energy stable in a similar way to \cite{du2018stabilized}. 

Adapting the techniques outlined here to equation \eqref{Eq:Marra} comes with some interesting challenges.
The main difficulty is extending the choice of mobility splitting \eqref{Eq:Mobility_Split} to the mobility \eqref{Eq:Marra_Mobility}.
This splitting is crucial in the proof of the bound preservation property, $|\rho| \leq 1$. 
From \cite{lyons2024phase}, solutions to \eqref{Eq:Marra} uphold the bound $0\leq |m| \leq \phi \leq 1$, and so, the chosen splitting technique must be devised so that this bound is preserved.
Additionally, as the reference system  for \eqref{Eq:Marra}, $g$, is now dependent on two values, it is not immediately clear how the time step should be split among its convex components in order to preserve the energy stability of the system. We aim to tackle these problems in a future work.

\section*{Acknowledgements} 

RL and AM gratefully acknowledge the financial support of Carl Tryggers Stiftelse via the grant CTS 21:1656. 
AM and GN express their deep appreciation for the financial support by the Knowledge Foundation (project nr. KK 2020-0152).

\bibliographystyle{plain}
\bibliography{morpho}
\end{document}